\documentclass[12pt]{amsart}
\usepackage{enumerate,amsmath,amssymb,bm}
\usepackage{url}
\usepackage{xspace}

\usepackage[margin=1in]{geometry}

\DeclareMathOperator{\pnt}{\raise 0.5mm \hbox{\large\textbf{.}}}

\newcommand{\note}[2][ ]{}

\newtheorem{theorem}{Theorem}
\newtheorem{lemma}[theorem]{Lemma}

\newtheorem{corollary}[theorem]{Corollary}
\newtheorem{conjecture}[theorem]{Conjecture}

\theoremstyle{definition}
\newtheorem{definition}[theorem]{Definition} 
\newtheorem{remark}[theorem]{Remark}

\usepackage[pdfauthor={William J. keith and Fabrizio Zanello},
            pdftitle={Parity of certain regular partition numbers},
            pdfsubject={enumerative combinatorics},
            pdfkeywords={regular partitions, parity},
            pdfproducer={Latex with hyperref},
            pdfcreator={latex->dvips->ps2pdf},
            pdfpagemode=UseNone,
            bookmarksopen=false,
            bookmarksnumbered=true]{hyperref}

\begin{document}
\title[Parity of the coefficients of certain eta-quotients]{Parity of the coefficients of certain eta-quotients}
\author{William J. Keith and Fabrizio Zanello} \address{Department of Mathematical Sciences\\ Michigan Tech\\ Houghton, MI  49931-1295}
\email{wjkeith@mtu.edu, zanello@mtu.edu}

\maketitle

\begin{abstract} We investigate the parity of the coefficients of certain eta-quotients, extensively examining the case of $m$-regular partitions. Our theorems concern the density of their odd values, in particular establishing lacunarity modulo 2 for specified coefficients; self-similarities modulo 2; and infinite families of congruences in arithmetic progressions.  For all $m \leq 28$, we either establish new results of these types where none were known, extend previous ones, or conjecture that such results are impossible.

All of our work is consistent with a new, overarching conjecture that we present for arbitrary eta-quotients, greatly extending Parkin-Shanks' classical conjecture for the partition function.  We pose several other open questions throughout the paper, and conclude by suggesting a list of specific research directions for future investigations in this area.
\end{abstract}

\section{Introduction and Statement of the Main Conjecture}

Eta-quotients are fundamental objects in the theory of modular forms and the combinatorics of partitions, widely studied for their arithmetic properties.  If we let $f_j := \prod_{i=1}^\infty (1-q^{ji})$, then an \emph{eta-quotient} is a function of the form
\begin{equation}\label{eta}
F(q)=q^{(1/24) (\sum \alpha_i r_i - \sum \gamma_i s_i)}\frac{\prod_{i=1}^uf_{\alpha_i}^{r_i}}{\prod_{i=1}^tf_{\gamma_i}^{s_i}},
\end{equation}
for integers $\alpha_i$ and $\gamma_i$ positive and distinct, $r_i, s_i > 0$, and $u,t\ge 0$.  The related function of most combinatorial interest typically does not present the power shift, and when not calculating precisely we often conflate the two.

One of the most challenging and interesting questions regarding eta-quotients, which has motivated a great deal of ongoing research in this area, concerns the study of the parity of their coefficients.  When we say a series has density $\delta$, we are speaking of the natural density of its odd coefficients.  We define it formally:
\begin{definition} The power series $f(q) = \sum_{n=0}^\infty a(n) q^n$ is said to have \emph{(odd) density} $\delta$  if the limit
$$\lim_{x \rightarrow \infty} \frac{1}{x} \# \{ a(n) : n \leq x,{\ } a(n)\equiv 1{\ } (\text{mod}{\ } 2) \}$$ exists and equals $\delta$.
\end{definition}
When $\delta=0$, we say that $f$ (or its coefficients) is \emph{lacunary modulo 2}. 

The best-known special case of (\ref{eta}) is perhaps
$$\frac{1}{f_1}=\sum_{n\ge 0} p(n)q^n,$$
where $p(n)$ is the ordinary partition function \cite{Andr}. It is widely believed that $p(n)$ has odd density $1/2$, and that, in fact, its parity behaves essentially \lq \lq at random'' \cite{Calk,NiSa,PaSh}, though proving this appears extremely hard. We know that $p(n)$ assumes infinitely many even, and infinitely many odd values.  By results of Ono \cite{Ono2} and Radu \cite{Radu}, this remains true for $p(An+B)$, for any given $A$ and $B$; in particular, $p(n)$ is not constant modulo 2 over any arithmetic progression.  We only remark here that this is in stark distinction to the existence of progressions $p(An+B) \equiv 0 \pmod{m}$ for any $m$ coprime to 6, for which there is a well-established unifying theorem in terms of the crank statistic (see, e.g., \cite{Mahlburg} for details).

The best asymptotic bound available today on the parity of $p(n)$ (which can also be applied to a broader class of eta-quotients) was obtained in 2018 by Bella\"iche \emph{et al.} \cite{BGS}, after decades of incremental progress. Unfortunately, it only guarantees $p(n)$ is odd for at least the order of $\sqrt{x}/ \log \log x$ integers $n\le x$, while the correct asymptotic estimate is conjecturally $x/2$.

More generally, we (along with then-graduate student S. Judge) recently conjectured that for any odd integer $t\ge 1$, the $t$-multipartition function $p_t(n)$, whose generating function is $1/f_1^t$, has odd density $1/2$ (see \cite{JKZ,JZ}, where we also related such densities for infinitely many values of $t$). 

On the other hand, despite an impressive body of literature and a number of mostly isolated conjectures, no eta-quotient is known today whose coefficients have positive odd density. In fact, to the best of our knowledge, not even the \emph{existence} of such density has ever been proven for any eta-quotient, except in those cases where it equals zero! 

Density zero can be established for some families of eta-quotients by means of the following very recent result by Cotron \emph{et al.} \cite{CMSZ}, who built on and generalized seminal works by Serre \cite{Serre} and then Gordon-Ono \cite{GO}. We phrase their theorem as follows:
\begin{theorem}\label{cot}(\cite{CMSZ}, Theorem 1.1.)
Let $F(q)=\frac{\prod_{i=1}^uf_{\alpha_i}^{r_i}}{\prod_{i=1}^tf_{\gamma_i}^{s_i}}$, and assume that
$$\sum_{i=1}^u \frac{r_i}{\alpha_i} \ge \sum_{i=1}^t s_i\gamma_i.$$
Then the coefficients of $F$ are lacunary modulo 2.
\end{theorem}

Note that the hypothesis of this theorem does not apply to the partition function, nor to many of the closely related functions we study in this paper, namely the $m$-\emph{regular partitions} $$B_m(q) := \frac{f_m}{f_1} = \sum_{n=0}^\infty b_m(n) q^n,$$ for $m \leq 28$.  However, after certain easy equivalences, Theorem \ref{cot} can be shown to immediately imply lacunarity modulo 2 for $B_m$ when $m \in \{ 2, 4, 8, 12, 16, 24\}$.

We denote by $\delta_m$ the density, should it exist, of $B_m$.

\begin{remark} Cotron \emph{et al.} \cite{CMSZ} suggested, without making it a formal conjecture, that when the hypothesis of Theorem \ref{cot} does not hold for an eta-quotient, then its density should be $\frac{1}{2}$. In this paper, we exhibit several counterexamples to that claim.  It already follows from earlier work \cite{HS1,JKZ} that $\delta_5 = \frac{1}{4} \delta_{20}$ (assuming they both exist), and therefore $\delta_5 \leq \frac{1}{4}$. In fact, numerical evidence seems to suggest that $\delta_5 = \frac{1}{8}$. Similarly, always assuming existence, we have $\delta_7 = \frac{1}{2}\delta_{28}$ (see \cite{JKZ}), and a natural guess is that $\delta_7=\frac{1}{4}$.

On the other hand, we note that the example at the end of \cite{CMSZ} is consistent with our data. It can actually be shown that the odd density of the function $G_{1,18,3}$ studied in section 5.1 of \cite{CMSZ} coincides with the odd density $\delta_6$ of the 6-regular partitions, and a computer search suggests that $\delta_6=\frac{1}{2}$.
\end{remark}

Almost all of the theorems we prove in this paper concern $B_m$ with $m$ odd, but we briefly note here, for $m$ even, that $b_2(n)$ and $b_4(n)$ are odd precisely for the pentagonal and triangular numbers, respectively. The functions $b_8$ and $b_{16}$ were studied by Cui and Gu in 2013 \cite{CuiGu}, who found families of even progressions $b_8(An+B)$ with $A = p^{2\alpha}$ for $p \equiv -1 \pmod{6}$, and $b_{16}$ with $p \equiv -1 \pmod{4}$.  The series for $b_{12}$ and $b_{24}$ are lacunary modulo 2 by Theorem \ref{cot} above (use $f_{2^j m}/f_1 \equiv {f_m}^{2^j}/f_1$), and we conjecture the existence of infinitely many even arithmetic progressions also for these $m$. However, in analogy to the case of the partition function $p(n)$, for all other even values of $m \leq 28$ we believe that no even progressions exist, and that the corresponding densities $\delta_m$ all equal $\frac{1}{2}$ (see Conjecture \ref{noncong}).

As for the odd values of $m$, we see in this paper that, in many interesting instances, $B_m$ exhibits multiple arithmetic progressions in which $b_m(An+B) \equiv 0 \pmod{2}$ for all $n$.  A few of these have been studied in the past; here, we produce theorems which exhibit even progressions -- in several cases, a large infinite family of them -- for all $m\le 28$ for which few or no such progressions were known before, with the only two exceptions $m=15$ and $m=27$, where we conjecture that no even progression exists. For certain values of $m$, we also include broader conjectures on the existence of suitable progressions or on their density.

The table below, while by no means an exhaustive survey of the area, is a substantial sampling of previously published results concerning the parity of $m$-regular partitions for $m$ odd.  Symbols $(\frac{n}{p})$ are Legendre symbols.

\begin{center}
\begin{tabular}{|c|c|c|}
\hline $m$ & $b_m(An+B)$ known to be even &  Source \\
\hline 3 & & \textbf{New in this paper} \\
\hline 5 & $b_5(2n)$, when $n$ is not twice a pentagonal number & Calkin \emph{et al.} \cite{Calkinetal} \\
\hline 7 & Family modulo $2p^2$ if $(\frac{-14}{p}) = -1$, $p$ prime  & Baruah and Das \cite{BD} \\
\hline 9 & $b_9(2^jn+ c(j))$ for various $j$  & Xia and Yao \cite{xiayao9} \\
 & & \textbf{More new in this paper} \\
\hline 11 & $b_{11}(22n+2,8,12,14,16)$ (finite family) &  Zhao, Jin, and Yao \cite{ZJY} \\
\hline 13 & $b_{13}(2n)$, when $n \neq k(k+1)$, $n \neq 13k(k+1)+3$  & Calkin \emph{et al.} \cite{Calkinetal} \\
\hline 15 & & \textbf{We conjecture none} \\
\hline 17 & $b_{17}(2 \cdot 17^{2\alpha+2}p^{2\beta}n+ c(p,\alpha,\beta))$ if $(\frac{-51}{p})=-1$ & Zhao, Jin, and Yao \cite{ZJY} \\
\hline 19 & $b_{19}(38n+2,8,10,20,24,28,30,32,34)$ (finite family) & Radu and Sellers \cite{RaduSellers} \\
 & & \textbf{More new in this paper} \\
\hline 21 & & \textbf{New in this paper} \\
\hline 23 & Family modulo $2p^2$ if $(\frac{-46}{p}) = -1$, $p$ prime & Baruah and Das \cite{BD} \\
\hline 25 & $b_{25}(100n+64,84)$ (finite family) & Dai \cite{Dai} \\
 & & \textbf{More new in this paper} \\
\hline 27 & & \textbf{We conjecture none} \\
\hline 
\end{tabular}
\end{center}
{\ }\\
\indent All of our theorems, as well as previous results on $m$-regular partitions, are consistent with the following very broad (and bold!) main conjecture on the parity of eta-quotients.  In fact, Conjecture \ref{mainconj} seems to be the first to appear in the literature in this level of generality.
\begin{conjecture}\label{mainconj}
Let $F(q)=\sum_{n\ge 0} c(n)q^n$ be an eta-quotient, shifted by a suitable power of $q$ so powers are integral, and denote by $\delta_F$ the odd density of its coefficients $c(n)$. We have:

i) For any $F$, $\delta_F$ exists and satisfies $\delta_F\le 1/2$.

ii) If $\delta_F= 1/2$, then for any nonnegative integer-valued polynomial $P$ of positive degree, the odd density of $c(P(m))$ is $1/2$. (In particular, $c(Am+B)$ has odd density $1/2$ for all arithmetic progressions $Am+B$.)

iii) If $\delta_F< 1/2$, then the coefficients of $F$ are identically zero (mod 2) on some arithmetic progression. (Note that this is not even \emph{a priori} obvious when $\delta_F=0$.)

iv) If the coefficients of $F$ are not identically zero (mod 2) on \emph{any} arithmetic progression, then they have odd density $1/2$ on \emph{every} arithmetic progression; in particular, $\delta_F= 1/2$.\\(Note that i), ii), and iii) together imply iv), and that iv) implies iii).)
\end{conjecture}

In Section \ref{listing}, we list our main theorems.  In Section \ref{back} we present the preliminary results and modular form tools we use throughout.  In Sections \ref{threesec} through \ref{twentyfivesec}, categorized by the value of $m$, we prove our theorems.  In the final section, we offer some concluding remarks, including applications to other classes of eta-quotients, and present further open questions that might be of interest for future investigations.

\section{Main Theorems and Further Conjectures}\label{listing}

We now introduce the main results and conjectures contained in the next sections. In addition to being consistent with the main Conjecture \ref{mainconj}, the theorems below generally describe more specific behaviors for each case listed.  We remark that, of course, they are not exhaustive of all candidates for even arithmetic progressions.

In order to avoid enormous repetition, we use the convention that any congruence $a \equiv b$ is modulo 2 unless explicitly stated otherwise.\\
\\
\noindent \emph{Case $m = 3$:} 

\begin{theorem}\label{b3thm} The series $b_3(2n)$ is lacunary modulo 2.  If $p \equiv 13, 17, 19, 23 \pmod{24}$ is prime, then $$b_3(2(p^2n + kp - 24^{-1})) \equiv 0 $$ for $1 \leq k \leq p-1$, where $24^{-1}$ is taken modulo $p^2$.
\end{theorem}

More generally, we conjecture that:
\begin{conjecture}\label{threeconj} For any prime $p > 3$, let $\alpha \equiv -24^{-1} \pmod{p^2}$, $0 < \alpha < p^2$.  It holds for a positive proportion of primes $p$, including those in the previous theorem, that $$\sum_{n=0}^\infty b_3(2(pn +\alpha)) q^n \equiv \sum_{n=0}^\infty b_3(2n) q^{pn}.$$ 

As a consequence, for every prime $p$ for which the conjecture holds, congruences similar to those of Theorem \ref{b3thm} hold as part of infinite families of congruences modulo increasing powers of $p^2$.
\end{conjecture}


For instance, one specific case of Conjecture \ref{threeconj} is the following (corresponding to $p=13$): 
\begin{theorem}\label{threesim} It holds that $$\sum_{n=0}^\infty b_3(26n+14) q^n \equiv \sum_{n=0}^\infty b_3(2n) q^{13n},$$ and therefore $$b_3(2 \cdot 13^2 n  +40) \equiv 0,$$ and by iteration, $$b_3 \left( 2 \cdot 13^{2k} n + 13^{2k-2} \cdot 40 + 14 \left( (13^{2k-2} - 1)/168 \right) \right) \equiv 0$$
for all $k\ge 1$.
\end{theorem}
{\ }\\
\noindent \emph{Case $m = 9$:}

\begin{theorem}\label{b9thm} For $p \in \{17, 19, 37\}$, let $\alpha \equiv -3^{-1} \pmod{2p}$, $0 < \alpha < 2p$, and $\beta = \lfloor \frac{2p}{3} \rfloor$. Then $$\sum_{n=0}^\infty b_9(2pn + \alpha) q^n \equiv q^{\beta} \sum_{n=0}^\infty b_9(2n+1)q^{pn}.$$
\end{theorem}

An example of this theorem (corresponding to $p=19$) is that $$\sum_{n=0}^\infty b_9(38n+25)q^n \equiv q^{12} \sum_{n=0}^\infty b_9(2n+1)q^{19n}.$$

Because $-3^{-1} \pmod{2p$} is always an odd integer, note that Theorem \ref{b9thm} implies several self-similarities between the sequence $\{b_9(2n+1)\}$ and certain of its subprogressions (see Corollary \ref{b9cor}).

In fact, we conjecture that an analogous phenomenon is true in a much greater level of generality: 
\begin{conjecture}\label{b9conj} Theorem \ref{b9thm} and Corollary \ref{b9cor} hold for all primes $p \equiv \pm 1 \pmod{9}$.
\end{conjecture}
{\ }\\
\noindent \emph{Case $m=19$:}

Some congruences of the form $b_{19}(38n+k) \equiv 0 $ were previously known \cite{RaduSellers}.  We have the following new family: 
\begin{theorem}\label{b19thm} It holds that $$\sum_{n=0}^\infty b_{19}(10n+8) q^n \equiv q \sum_{n=0}^\infty b_{19}(2n)q^{5n}.$$  This implies that $$b_{19}(50n + 10k + 8) \equiv 0 $$
for all $k \not\equiv 1 \pmod{5}$.  By iteration, $$b_{19}(2n) \equiv b_{19}\left(2 \cdot 5^{2d} n + 9 ((5^{2d}-1)/24) \right),$$ and therefore $$b_{19} \left( 2 \cdot 5^{2d} (50n+10k+8) + 9 ( (5^{2d}-1)/24) \right) \equiv 0, $$ for all $d,k\ge 1$ with  $k \not\equiv 1 \pmod{5}$.
\end{theorem}

The next examples up in this family are the progressions $b_{19}(1250n + j) \equiv 0 $ for $j \in \{ 218, 718, 968, 1218 \}$.

Similarly to the situation we encountered with $b_9$, we conjecture that the previous theorem is the first case (corresponding to the prime $p=5$) of an infinite family, even though here we were unable to characterize the primes involved. We have:
\begin{conjecture}\label{b19conj} For a prime $p>3$, let $\alpha \equiv -3 \cdot 8^{-1} \pmod{p}$, $0 < \alpha < p$, and set $\beta = \lfloor \frac{3p}{8} \rfloor$.  Then, for a positive proportion of primes $p$, it holds that: $$\sum_{n=0}^\infty b_{19}(2pn + 2\alpha) q^n \equiv q^\beta \sum_{n=0}^\infty b_{19}(2n) q^{pn} .$$
\end{conjecture}

Note that, by iteration, the congruences of Conjecture \ref{b19conj} imply the existence of infinite even families analogous to those of Theorem \ref{b19thm}.\\
{\ }\\
\noindent \emph{Case $m = 21$:}

We establish lacunarity modulo 2 for the sequences $b_{21}(4n)$ and $b_{21}(4n+1)$, and give two theorems identifying even progressions within these sequences: 
\begin{theorem}\label{21in40} If $p \equiv 19, 37, 47, 65, 85, 109, 113, 115, 137, 139, 143,$ or $167 \pmod{168}$ is prime, then $$b_{21}(4(p^2 n + kp - 5 \cdot 24^{-1})) \equiv 0 $$ for all $1 \leq k < p$, where $24^{-1}$ is taken modulo $p^2$.
\end{theorem}

As an example, $$b_{21}(4\cdot19^2n + 76k+300) \equiv 0,$$ for all $k=1,2, \dots,18$.
\begin{theorem}\label{21in41} If $p \equiv 13, 17, 19,$ or $23 \pmod{24}$ is prime, then $$b_{21}(4(p^2n+kp-11 \cdot 24^{-1})+1) \equiv 0 $$ for all $1 \leq k < p$, where $24^{-1}$ is taken modulo $p^2$.
\end{theorem}

As an example, $$b_{21}(4 \cdot 13^2 n + 52k + 309) \equiv 0,$$ for all $k=1,2, \dots,12$.\\
\\
\noindent \emph{Case $m = 25$:}

We have results within each residue class modulo 4 for $B_{25}$: 
\begin{theorem}\label{b25thm} The sequences $b_{25}(4n)$, $b_{25}(4n+1)$, and $b_{25}(8n+3)$ are not identically zero but lacunary modulo 2, and the following congruences hold within arithmetic progressions modulo 4: \begin{enumerate}
\item If $p \equiv 31$ or $39 \pmod{40}$ is prime, then $$b_{25}(4(p^2n + p k - 4^{-1})) \equiv 0 $$ for all $1 \leq k < p$, where $4^{-1}$ is taken modulo $p^2$.  
\item If $p \equiv 31$ or $39 \pmod{40}$ is prime, then $$b_{25}(4(p^2n+pk - 2^{-1}) + 1) \equiv 0 $$ for all $1 \leq k < p$, where $2^{-1}$ is taken modulo $p^2$.
\item $$b_{25}(40n+34) \equiv 0 .$$
\item $$\sum_{n=0}^\infty b_{25}(8n+3) \equiv f_1^{12} + q^{12} f_{25}^{12}.$$
Therefore, $b_{25}(8n+3)$ is even as long as $n$ is not 4 times a triangular number or 12 more than 100 times a triangular number. This further implies evenness for infinitely many arithmetic progressions, the simplest of which is $$b_{25}(16n+11) \equiv 0 .$$
\end{enumerate}
\end{theorem}

We also conjecture the following noncongruences, analogous to Ono's and Radu's theorems for the partition function (see \cite{Ono2,Radu}):
\begin{conjecture}\label{noncong} Let $m \in \{ 6, 10, 14, 15, 18, 20, 22, 26, 27, 28\}$. We have: 
\begin{enumerate}
\item For no integers $A>0$ and $ B \ge 0$, $b_m(An+B) \equiv 0 $ for all $n$.
\item The series $f_m / f_1$ has odd density $1/2$.
\end{enumerate}
\end{conjecture}

Note that if the main Conjecture \ref{mainconj} holds, then either clause of Conjecture \ref{noncong} implies the other (and much more).

Finally, in the last section of our paper, we also demonstrate that an odd density may be constant over suitable infinite families of eta-quotients. As an example, we show:
\begin{theorem}\label{b6thm} Assume the odd density $\delta^{(k)}$ of $$\frac{{f_3}^{9k+2}}{{f_1}^{3k+1}}$$ exists for \emph{any} $k \geq 0$. Then $\delta^{(k)}$ exists for \emph{all} $k \geq 0$, and its value is independent of $k$.  In particular, since $\delta^{(0)} = \delta_6$ (the odd density of the 6-regular partitions), if $\delta_6$ exists then all of the $\delta^{(k)}$ exist and are equal to $\delta_6$. 
\end{theorem}

\section{Background}\label{back}

Given two power series $f(q) = \sum_{n=0}^\infty a(n) q^n$ and $h(q) = \sum_{n=0}^\infty b(n) q^n$, if we say $f(q) \equiv h(q)$, in this paper we always mean that $a(n) \equiv b(n) \pmod{2}$ for all $n$.

Two common series expansions we extensively need (see for instance \cite{Andr}) are the modulo 2 reduction of Euler's Pentagonal Number Theorem, $$f_1 \equiv \sum_{n \in \mathbb{Z}} q^{\frac{n}{2}(3n-1)},$$ \noindent so named since the exponents are the eponymous (generalized) pentagonal numbers, and the classical fact that the cube of $f_1$ has odd coefficients precisely at the triangular numbers, $${f_1}^3 \equiv \sum_{n=0}^\infty q^{\binom{n+1}{2}}.$$

A fact we also use frequently is:
\begin{lemma}\label{222} For any power series $f(q) = \sum_{n=0}^\infty a(n) q^n$, it holds that $$(f(q))^2 \equiv \sum_{n=0}^\infty a(n) q^{2n}.$$
In particular, ${f_j}^2 \equiv f_{2j}$.
\end{lemma}

The eta-quotients ${f_t}^t / f_1$ (which are the generating functions for the so-called $t$-core partitions) have the following modulo 2 expansions for $t=3$ and $t=5$ (see \cite{Robbins}, Theorem 7, and \cite{Hirsch}, equation (10), respectively):
\begin{lemma}\label{corelemma} We have:
\begin{equation}\label{threecore} \frac{{f_3}^3}{f_1} \equiv \sum_{n \in \mathbb{Z}} q^{n(3n-2)};
\end{equation}

\begin{equation}\label{fivecore}
\frac{{f_5}^5}{f_1} \equiv \sum_{n=1}^\infty q^{n^2-1} + \sum_{n=1}^\infty q^{2n^2-1} + \sum_{n=1}^\infty q^{5n^2-1} + \sum_{n=1}^\infty q^{10n^2-1} .
\end{equation}
\end{lemma}

The following classical result can be attributed to Landau \cite{Landau}:
\begin{lemma}\label{laclemma} Let $r(n)$ and $s(n)$ be quadratic polynomials.  Then $$\left( \sum_{n \in \mathbb{Z}} q^{r(n)} \right) \left( \sum_{n \in \mathbb{Z}} q^{s(n)} \right)$$ \noindent is lacunary modulo 2.
\end{lemma}

\subsection{Modular forms background}

We include here sufficient background from the theory of modular forms in order to make our proofs relatively self-contained.  For further details, see for instance the standard text \cite{Koblitz} or the monograph \cite{Ono}.

With $q = e^{2 \pi i z}$, define the \emph{$\eta$-function} as $$\eta(z) :=  q^{1/24} f_1.$$  An eta-quotient is a {modular form} for $\Gamma_0(N)$ of weight $k$ and character $\chi$ under the following conditions on the exponents and arguments appearing, by a theorem of Gordon-Hughes \cite{GH} and Newman \cite{Newman}, which also specifies the necessary values of $N$, $k$, and $\chi$.
\begin{theorem} [\cite{GH,Newman}]\label{ghn} Let $f(z) = \prod_{\delta \vert N} \eta^{r_\delta} (\delta z)$, with $r_\delta \in \mathbb{Z}$.  If
$$\sum_{\delta \vert N} \delta r_\delta \equiv 0 \, \pmod{24} \quad \text{ and } \quad \sum_{\delta \vert N} \frac{N}{\delta} r_\delta \equiv 0 \, \pmod{24},$$
then $f(z)$ is a modular form for $\Gamma_0(N)$ of weight $k=\frac{1}{2} \sum r_\delta$ and character $\chi(d) = \left( \frac{(-1)^k s}{d} \right)$, where $s = \prod_{\delta \vert N} \delta^{r_\delta}$ and $\left(\frac{\cdot}{d}\right)$ denotes the Kronecker symbol.
\end{theorem}

\begin{remark} Note that the Kronecker symbol appearing in Theorem \ref{ghn} (as described in \cite{OnoKronecker}, Lemma 6 and \cite{XiongKronecker}, Proposition 2.1), multiplicatively extends the Jacobi symbol from odd values to all positive integers, by the assignment:
 \begin{equation}
    \left( \frac{a}{2} \right) =
    \begin{cases}
      0 & \text{if } \, a \text{ even}; \\
      1 & \text{if } \, a \equiv \pm 1\pmod{8}; \\
      -1 & \text{if } \, a \equiv \pm 3\pmod{8}.
    \end{cases}
  \end{equation}
\end{remark}

Given two modular forms of weight $k$ and characters $\chi$ and $\psi$ for $\Gamma_0(N)$, the following result of Sturm \cite{sturm} gives a criterion -- crucially, one involving a finite amount of calculation -- to determine when all of their coefficients are congruent modulo a given prime.
\begin{theorem} [\cite{sturm}]\label{sturmthe} Let $p$ be a prime number, and $f(z) = \sum_{n=n_0}^\infty a(n) q^n$ and $g(z) = \sum_{n=n_1}^\infty b(n) q^n$  be holomorphic modular forms of weight $k$ for $\Gamma_0(N)$ of characters $\chi$ and $\psi$, respectively, where $n_0,n_1\ge 0$.  If either $\chi=\psi$ and 
$$a(n) \equiv b(n) \pmod{p} \quad \text{for all} \quad n\le \frac{kN}{12}\cdot \prod_{d{\ }\emph{prime};{\ } d \vert N} \left(1+\frac{1}{d}\right),$$
or $\chi \neq \psi$ and
$$a(n) \equiv b(n) \pmod{p} \quad \text{for all} \quad n\le \frac{kN^2}{12}\cdot \prod_{d{\ }\emph{prime};{\ } d \vert N} \left(1-\frac{1}{d^2}\right),$$
then $f(z) \equiv g(z) \pmod{p}$ (i.e., $a(n) \equiv b(n) \pmod{p}$ for all $n$).
\end{theorem}

We may sometime refer to the bound calculated above as the \emph{Sturm bound}.

In order to use Sturm's Theorem we require that the modular forms being compared be holomorphic, i.e., vanishing to nonnegative orders at the cusps, a representative set of rational values.  These orders are determined by the following theorem of Ligozat:
\begin{theorem}[\cite{Ligo,Ono}]\label{ligozat} Let $c$, $d$ and $N$ be integers such that $d \vert N$ and $\gcd(c,d)=1$.  Then if $f$ is an $\eta$-quotient satisfying the conditions of Theorem \ref{ghn} for $\Gamma_0(N)$, the order of $f$ at the cusp $\frac{c}{d}$ is
$$\frac{N}{24} \sum_{\delta \vert N} \frac{\left(\gcd(d,\delta)\right)^2 r_\delta}{\gcd \left(d,\frac{N}{d}\right) d \delta}.$$
\end{theorem}

Given a prime $p$, the \emph{Hecke operator} $T_p$ (in fact, $T_{p,k,\chi}$, but the latter two values are usually suppressed when they are clear from context) is a $\mathbb{C}$-vector space endomorphism on $M_k(\Gamma_0(N),\chi)$.  For $$f(z) = \sum_{n=0}^\infty a(n) q^n \in M_k(\Gamma_0(N),\chi),$$ the action of this operator is
$$f \vert T_p := \sum (a(pn) + \chi(p) p^{k-1}a(n/p)) q^n,$$
where we take by convention that $a(n/p) = 0$ whenever $p \nmid n$.  

If $f$ is an $\eta$-quotient with the properties listed in Theorem \ref{ghn}, and $p \vert s$ as defined above, then $\chi(p) = 0$ so that the latter term vanishes.  In this case, we have the factorization property that $$\left( f \cdot \sum_{n=0}^\infty g(n) q^{pn} \right) \vert T_p = \left( \sum_{n=0}^{\infty} a(pn)q^n \right) \left(\sum_{n=0}^\infty g(n) q^n \right).$$

Another operator we will need is the \emph{$V$ operator}.  Given $f(z) = \sum_{n=n_0}^\infty a(n) q^n \in M_k(\Gamma_0(N),\chi)$ and a positive integer $d$, define $$f(z) \vert V(d) := \sum_{n=n_0}^\infty a(n) q^{dn}.$$ 

Then it holds (see \cite{Ono}, Proposition 2.22) that $$f(z) \vert V(d) \in M_k(\Gamma_0(Nd),\chi).$$

\section{Case $m=3$}\label{threesec}

In this section, we present the proof of Theorem \ref{b3thm}, stated in Section \ref{listing}.

\begin{proof}[Proof of Theorem \ref{b3thm}]  A crucial tool, which we shall use in many variations, is the following identity (\cite{JKZ}, equation (12)):
\begin{equation}\label{onethree}
(f_1 f_3)^3 \equiv {f_1}^{12} + q {f_3}^{12}.
\end{equation}

Dividing through by ${f_1}^4 {f_3}^2$, we obtain an even-odd dissection of the 3-regular partitions:
\begin{align}\label{3reg}
\frac{f_3}{f_1} &\equiv \frac{f_1^8}{f_3^2} + q \frac{f_3^{10}}{f_1^4}.
\end{align}

By the identity ${f_1}^2 \equiv f_2$ (see Lemma \ref{222}), on which we will rarely further remark explicitly to avoid heavy repetition, the first term on the right side contributes nonzero terms modulo 2 only to even powers of $q$, while the second term only to odd powers.  

Thus, by extracting those even powers and halving them, we obtain: $$\sum_{n=0}^\infty b_3(2n) q^n \equiv \frac{{f_1}^4}{f_3}.$$

A simple algebraic analysis shows the following identity which 3-dissects the triangular numbers (see \cite{Hirsch2}, Chapter 1): 
\begin{equation}\label{onecubed}
{f_1}^3 \equiv f_3 + q{f_9}^3.
\end{equation}

From this and Lemma \ref{corelemma}, we get:
$$\frac{{f_1}^4}{f_3} \equiv f_1 + q f_1 \frac{{f_9}^3}{f_3} $$$$\equiv f_1 \left( 1 + q \sum_{n \in \mathbb{Z}} q^{3n(3n-2)} \right) \equiv f_1 \left( 1 + \sum_{n \in \mathbb{Z}} q^{(3n-1)^2} \right).$$

The second congruence follows from the fact that $\frac{{f_9}^3}{f_3}$ is simply the 3-core generating function $\frac{{f_3}^3}{f_1}$ under the substitution $q \rightarrow q^3$, which we will henceforth refer to as being \emph{magnified} by 3.

Therefore, $b_3(2n)$ can only be odd if $n$ is the sum of a pentagonal number and either zero or a square not divisible by 3.  Thus, as a product of two quadratic forms, Lemma \ref{laclemma} gives lacunarity modulo 2 for $b_3(2n)$. We now proceed to identify arithmetic subprogressions on which this is identically zero modulo 2.

Modulo any prime $p^2$ with $p>3$, completing the square yields that the pentagonal numbers are of the form $$\frac{3}{2}x^2 - 24^{-1},$$ where $2^{-1}$ and $24^{-1}$ are taken modulo $p^2$.  That is, let $p > 3$ be prime.  Then we have the following:
\begin{align*} n &\equiv \frac{k}{2}(3k-1) \pmod{p^2} \\
 &\equiv (3 \cdot 2^{-1} ) (k^2 - 3^{-1} k) \pmod{p^2} \\
 &\equiv (3 \cdot 2^{-1} ) (k^2 - 3^{-1} k + 36^{-1}) - 24^{-1} \pmod{p^2} \\
 &\equiv (3 \cdot 2^{-1} ) (k - 6^{-1} )^2 - 24^{-1} \pmod{p^2}.
\end{align*}

As for squares not divisible by 3, we observe that the divisibility criterion is irrelevant when noting that  the sequences can eventually intersect any residue class modulo $p^2$ that is a quadratic residue, but not any quadratic nonresidue modulo $p^2$.

Let $a \equiv \frac{3}{2}x^2 - 24^{-1} \pmod{p^2}$ and $b \equiv y^2 \pmod{p^2}$, for $x, y \in \mathbb{Z}$, and write $n = a+b$.  

Now suppose $n \equiv kp - 24^{-1} \pmod{p^2}$.  Then $\frac{3}{2}x^2 + y^2 \equiv 0 \pmod{p}$.  This is possible if $x, y \equiv 0 \pmod{p}$, in which case $n \equiv -24^{-1} \pmod{p^2}$.  

Conversely, $x \not\equiv 0 \pmod{p}$ if and only if $y \not\equiv 0 \pmod{p}$, and if so, then we may write $$\frac{y^2}{x^2} \equiv \frac{-3}{2} \pmod{p^2}.$$
But this can only be true if $-6$ is a quadratic residue modulo $p$.  If not, then no such $x$ and $y$ can exist.

Therefore, if $(-6/p) = -1$, we have that $$b_3(2(p^2 n + kp - 24^{-1})) \equiv 0,$$ for all $n\ge 0$ and $1 \leq k \le p-1$. Since $-6$ is a quadratic nonresidue precisely for the primes $p \equiv 13, 17, 19, 23 \pmod{24}$, Theorem \ref{b3thm} is proved. \end{proof}

We next show the specific case of Conjecture \ref{threeconj} stated in Theorem \ref{threesim}.

\begin{proof}[Proof of Theorem \ref{threesim}] This proof employs a congruence argument for modular forms.  In order to prevent substantial repetition we give one full example of such an argument in the next section.  Here, we provide sufficient detail to reproduce the proof for this theorem.

We construct modular forms $$C_1(q) := q^6 \frac{{f_1}^4}{f_3} {f_{39}}^2 {f_{13}}^5 \equiv q^6 \left( \sum_{n=0}^\infty b_3(2n) q^n \right) {f_{39}}^2 {f_{13}}^5$$ and $$C_2(q) := q \frac{{f_{13}}^4}{f_{39}} {f_3}^2 {f_1}^5 \equiv q \left( \sum_{n=0}^\infty b_3(2n) q^{13n} \right) {f_3}^2 {f_1}^5.$$  We then calculate that $$C_1 \vert T_{13} \equiv q \left( \sum_{n=0}^\infty b_3(26n+14) q^n \right) {f_3}^2 {f_1}^5 .$$

We find that $C_1$, $C_1 \vert T_{13}$, and $C_2$ are all modular forms of weight 5, level 117, and character $\chi = \left( \frac{-39}{d} \right)$.  

Thus, we verify holomorphicity by the criterion of Theorem \ref{ligozat}, and apply Sturm's Theorem \ref{sturmthe} with bound 70 to determine, after a brief calculation, that the two forms $C_1 \vert T_{13}$ and $C_2$ are congruent in all coefficients.  Dividing out the factors $q {f_3}^2 {f_1}^5$, we establish the congruence of the desired sequences.

Since only powers for which $13 \vert n$ can be nonzero on the right side of the statement, we obtain: $$b_3(26(13n+1) + 14) = b_3(2 \cdot 13^2 n + 40) \equiv 0.$$

Finally, by repeatedly applying the relation $$b_3(2n) \equiv b_3(26 \cdot 13n + 14),$$
we get:
\begin{align*} b_3(2 \cdot 13^2 n + 40) &\equiv b_3(2\cdot 13^2(13^2n+20) + 14) \\
 &= b_3(2 \cdot 13^4 n +  13^2 \cdot 40 + 14) \\
 &\equiv b_3(2 \cdot 13^6 n + 13^4 \cdot 40 + 13^2 \cdot 14 + 14) \\
 &\equiv ... \\
 &\equiv b_3 \left(2 \cdot 13^{2k} n + 13^{2k-2}\cdot 40 + 14 ((13^{2k-2} - 1)/168) \right),
\end{align*}
where the last line is given by a finite geometric summation.
\end{proof}

\section{Case $m=9$}\label{ninesec}

\begin{proof}[Proof of Theorem \ref{b9thm}] The even-odd dissection of $B_9$ is given by \cite{xiayao9}, Lemma 2.1: $$\sum_{n=0}^\infty b_9(n) q^n = \frac{f_6^5}{f_4 f_{18}} + q \frac{f_2 f_9^4}{f_6}.$$

Thus,
$$\sum_{n=0}^\infty b_9(2n) q^n \equiv \frac{f_3^5}{f_2 f_9};$$
$$\sum_{n=0}^\infty b_9(2n+1) q^n \equiv \frac{f_1 f_9^2}{f_3}.$$

Each clause of Theorem \ref{b9thm} is proven by establishing a congruence of parity between the modular forms $C_p(q) \vert T_p$ and $$q^{r_p} \frac{f_p f_{9p}^2}{f_{3p}} f_1^{j_p}$$ for $p = 17$, 19, and 37, where $$C_p = q^{d_p} \left( \sum_{n=0}^\infty b_9(2n+1)q^n \right) f_p^{j_p},$$
and
$$(d_{17},d_{19},d_{37}) = (12,7,13);$$
\begin{equation}\label{drj}
(r_{17},r_{19},r_{37}) = (12,13,24);
\end{equation}
$$(j_{17},j_{19},j_{37}) = (16,8,8).$$

We will prove the theorem for $p=17$ in detail.  The calculations are similar for the other values of $p$.

We construct:
$$C_1(q) := q^{12} \left( \sum_{n=0}^\infty b_9(2n+1) q^n \right) f_{17}^{16} = \frac{\eta(z) \eta(9z)^2\eta(17z)^{16}}{\eta(3z)}.$$

By Theorem \ref{ghn}, we compute that this is an element of $M_9 (\Gamma_0(3^3 \cdot 17), \chi_1(d) )$, where $\chi_1(d) = \left( \frac{-27 \cdot 17^{16}}{d} \right)$.  This character is simply $\left( \frac{-3}{d} \right)$ for $17 \nmid d$, and 0 if $17 \vert d$.

We verify by Theorem \ref{ligozat} that the form vanishes to nonnegative order at all cusps, and therefore it is holomorphic.  Since the Hecke operator is an endomorphism on $M_k(\Gamma_0(N),\ \chi)$, we have that $$C_1 \vert T_{17} = q \left( \sum_{n=0}^\infty b_9(34n+11) q^n \right) {f_1}^{16} \in M_9(\Gamma_0(27 \cdot 17), \chi_1(d) )$$ \noindent as well.  By Theorem \ref{sturmthe}, the Sturm bound for this space of forms is 486 as long as we are comparing forms with the same character.  

For the right side of our comparison, we construct $$G_1(q) := q^{12} \frac{f_{17} {f_{153}}^2}{f_{51}} {f_1}^{16}.$$

We compute, again by Theorem \ref{ghn}, that this is also an element of $M_9(\Gamma_0(27 \cdot 17), \chi(d) )$, noting that the character is
$$\chi (d) = \left(\frac{-3^3 \cdot 17^2}{d}\right) =\chi_1(d).$$

Ligozat's Theorem \ref{ligozat} then confirms holomorphicity.

We wish to verify the congruence $$q \left( \sum_{n=0}^\infty b_9(34n+11) q^n \right) {f_1}^{16} \equiv q^{12} \frac{f_{17} {f_{153}}^2}{f_{51}} {f_1}^{16}.$$  The coefficient of $q^{486}$ on the left side involves the value $b_9(16535)$; thus, $f_9 / f_1$ must be expanded at least that far, and the product on the right side must be constructed up to the $q^{486}$ term. Finally, expansion with a calculation package such as \emph{Mathematica} confirms that all coefficients up to the desired bound are congruent, and the theorem is established.

The argument is similar for the other two primes, with the necessary values being listed in equations (\ref{drj}). \end{proof}

We can obtain several corollaries by observing that $$\frac{f_p f_{9p}^2}{f_{3p}}$$ is the $q \rightarrow q^p$ magnification of the odds in $B_9$. That is,
$$\frac{f_p{f_{9p}}^2}{f_{3p}} = \sum_{n=0}^\infty b_9(2n+1) q^{pn},$$
which clearly has coefficient 0 in any term $q^k$ with $p \nmid k$.  This self-similarity can thus establish families of progressions modulo powers of the primes involved, from which we deduce the following corollary:
\begin{corollary}\label{b9cor} For $p \in \{17, 19, 37\}$, let $\{ 2 p^{2k} n + B_k \}$ be the sequence obtained by iteratively taking $n \equiv \beta \pmod{p}$ in the prior sequence of index $k-1$, with $k=1$ being the sequence described on the right side of the congruence in Theorem \ref{b9thm}, in which $\beta$ and $\alpha$ are defined.  Then, for any $n \equiv pj + (\alpha-1)/2 \pmod{p^2}$ such that $j \not\equiv \beta \pmod{p}$, we have: $$b_9(2 p^{2k} n  + B_k) \equiv 0 .$$
\end{corollary}

As an example with base sequence $\{b_9(38n+25)\}$, if $n \not\equiv 12 \pmod{19}$, then $b_9(38n+25) \equiv 0 $. Thus,  for instance, $$b_9(38(19n+7)+25) = b_9(722n+291) = b_9(2(361n+145)+1) \equiv 0 .$$

Our base case $k=1$ is the sequence $$b_9(38(19n+12)+25)  = b_9(722n + 481) \equiv b_9(2n+1) .$$ By iterating, we obtain $$b_9(2n+1) \equiv b_9(722n+481) \equiv \dots \equiv b_9 \left( 2 \cdot 19^{2k} n + 480 ((19^{2k}-1)/360) + 1\right).$$

By taking now the $n \equiv 145 \pmod{361}$ subsequence, we obtain new families of even progressions, of which the following is the next example for $k=2$: $$b_9(722(361j+145) + 481) = b_9 (2 \cdot 19^4 j + 105171) \equiv 0.$$ The same process can be performed for the other progressions.

This yields an infinite collection of such congruences for each prime for which the theorem is proved.  In fact, interestingly, we believe that this self-similarity holds for all primes $p \equiv \pm 1 \pmod{9}$ (see Conjecture \ref{b9conj}).

\section{Case $m=19$}

\begin{proof}[Proof of Theorem \ref{b19thm}] We note that once the self-similarity of Theorem \ref{b19thm} is proved, the even progressions and the  iterated infinite families follow by the same argument that we employed for Corollary \ref{b9cor}.

Because we do not have a convenient even-odd dissection for $b_{19}$, the proof begins with a pair of Hecke operators. For the right side, we construct:
$$C_1(q) := q^{52} \frac{f_{19}}{f_1} f_5^{246} = \frac{\eta(19z) \eta(5z)^{246}}{\eta(z)}.$$
We compute that this is an element of $M_{123}(\Gamma_0(95),\chi_1)$, where
$$\chi_1(d) = \left( \frac{-19 \cdot 5^{246}}{d} \right) = \left(\frac{-19}{d}\right)$$
if $5 \nmid d$, and $\chi_1(d)=0$ if $5 \vert d$.

We now apply Hecke operators to construct $$G(q) := C_1(q) \vert T(2) \vert V(5) \equiv q^{130} \left( \sum_{n=0}^\infty b_{19} (2n) q^{5n} \right) f_{25}^{123}.$$
We verify that this is an element of $M_{123}(\gamma_0(19 \cdot 2 \cdot 5^2),\chi_1)$.

For the left side, we construct $$C_2(q) := q^{1282} \frac{f_{19}}{f_1} {f_{125}}^{246} = \frac{\eta(19z) \eta(125z)^{246}}{\eta(z)}.$$
We compute that this is an element of $M_{123}(\Gamma_0(19 \cdot 5^3), \chi_2)$, where
$$\chi_2(d) = \left( \frac{-19 \cdot 5^{738}}{d} \right) = \chi_1(d).$$

We now apply a different sequence of Hecke operators, to obtain: $$H(q) := C_2(q) \vert T(2) \vert T(5) \equiv q^{129} \left( \sum_{n=0}^\infty b_{19}(10n+8) q^n \right) {f_{25}}^{123},$$ \noindent and we verify that this is an element of $M_{123}(\Gamma_0(19 \cdot 2 \cdot 5^3), \chi_1)$.

Theorem \ref{ligozat} confirms that both initial forms are holomorphic, and the Hecke operators preserve this property.  Thus, $G(q)$ and $H(q)$ may be compared as elements of $M_{123}(\Gamma_0(19 \cdot 2 \cdot 5^3), \chi_1)$, for which the Sturm bound (Theorem \ref{sturmthe}) is given by:
$$\frac{123}{12} (19 \cdot 2 \cdot 5^3)  \left(1 +  \frac{1}{2} \right) \left(1+ \frac{1}{5} \right) \left(1+ \frac{1}{19} \right) = 92250.$$
Hence, in order to confirm the desired congruence, we must calculate terms of $b_{19}$ up to $q^{922508}$.

This was done over the course of several days on a desktop computer, and the theorem is confirmed.  The iterative results follow. \end{proof}

\begin{remark} The relatively large Sturm bound is the main reason we did not choose to confirm more cases from Conjecture \ref{b19conj}.\end{remark}

\section{Case $m= 21$}\label{twentyonesec}

In this section, we prove Theorems \ref{21in40} and \ref{21in41}.  We begin with some preliminaries.  In addition to identity (\ref{3reg}), we require the following, which is easily seen to be equivalent to \cite{Lin}, equation 2.4:
\begin{align}\label{7reg}
\frac{f_7}{f_1} &\equiv f_1^6 + q f_1^2 f_7^4 + q^2 \frac{f_7^8}{f_1^2}.
\end{align}

Magnifying equation (\ref{7reg}) by $q \rightarrow q^3$, we can now dissect $B_{21}$ as:
\begin{align*} \frac{f_{21}}{f_1} &= \frac{f_3}{f_1}\frac{f_{21}}{f_3} \equiv \left( \frac{f_1^8}{f_3^2} + q \frac{f_3^{10}}{f_1^4}\right) \left(  f_3^6 + q^3 f_3^2 f_{21}^4 + q^2 \frac{f_{21}^8}{f_3^2}\right).
\end{align*}

Extracting  terms that are even and odd, and then further extracting those that are, respectively, 0 and 1 modulo 4, we obtain the following:
$$\sum_{n=0}^\infty b_{21}(2n) q^n \equiv f_1^4f_3^2 + q^3 \frac{f_1^4 f_{21}^4}{f_3^2} + q^2 \frac{f_3^6 f_{21}^2}{f_1^2};$$
$$\sum_{n=0}^\infty b_{21}(4n) q^n \equiv f_1^2 f_3 + q \left( \frac{f_3^3}{f_1} \right) f_{21};$$
$$\sum_{n=0}^\infty b_{21}(2n+1) q^n \equiv \frac{f_3^8}{f_1^2} + q^3 \frac{f_3^4 f_{21}^4}{f_1^2} + q f_1^4 f_{21}^2;$$
$$\sum_{n=0}^\infty b_{21}(4n+1) q^n \equiv \left( \frac{f_3^3}{f_1} \right) f_3 .$$

Because, by Lemma \ref{corelemma},
$$\frac{f_3^3}{f_1} \equiv \sum_{n \in \mathbb{Z}} q^{n(3n-2)},$$
we deduce from Lemma \ref{laclemma} that both $\sum_{n=0}^\infty b_{21}(4n) q^n$ and $\sum_{n=0}^\infty b_{21}(4n+1)q^n$ are lacunary modulo 2.  

As for the two progressions 2 and 3 (mod 4), we conjecture that their behavior modulo 2 is entirely different. We have:
\begin{conjecture} The series $\sum_{n=0}^\infty b_{21}(4n+2)q^n$ and $\sum_{n=0}^\infty b_{21}(4n+3)q^n$ have odd density $1/2$. Consequently, the 21-regular partition function has odd density $1/4$.
\end{conjecture}

We now show that each of the two lacunary progressions modulo 2, expectedly, contain even arithmetic progressions. Note that with $b_{21}(4n)$ we have the additional complication that there are two terms to study.

\begin{proof}[Proofs of Theorems \ref{21in40} and \ref{21in41}] We comment in detail only on Theorem \ref{21in40}; the proof for Theorem \ref{21in41} is similar to previous arguments.

For each term individually in $\sum_{n=0}^\infty b_{21}(4n) q^n$, the argument is in part similar to the proofs in Section \ref{threesec}; the new task is to connect several of them. Note that 24 is invertible modulo all the primes $p$ listed in the theorem.

We consider the term $f_1^2 f_3$.  Residues $a$ modulo $p^2$ that this series can reach must satisfy $$3x^2-x + \frac{9}{2}y^2 - \frac{3}{2}y \equiv a \pmod{p^2}$$ \noindent for some $x, y \in \mathbb{Z}$.  Letting $b \equiv x - 6^{-1} \pmod{p^2}$ and $c \equiv y - 6^{-1} \pmod{p^2}$, we find that $$a + 5 \cdot 24^{-1} \equiv 3 b^2 + \frac{9}{2} c^2 \pmod{p^2}.$$

Thus if $a \equiv -5 \cdot 24^{-1} \pmod{p}$, we must have $3b^2 + \frac{9}{2} c^2 \equiv 0 \pmod{p}$.

Suppose this relation to hold.  If either  $b$ or $c$ is zero modulo $p$, then the other must be as well, and then $a \equiv -5 \cdot 24^{-1} \pmod{p^2}$.  If neither is zero modulo $p$, then it holds that $$(b/c)^2 \equiv -3/2 \pmod{p^2}.$$ Therefore, $-6$ must be a quadratic residue modulo $p$.

We have that $-6$ is a quadratic nonresidue for primes $p \equiv 13, 17, 19, 23 \pmod{24}$.  Thus, for these primes, the series $f_1^2 f_3$ avoids $k p - 5\cdot24^{-1} \pmod{p^2}$ for any $k \not\equiv 0 \pmod{p}$.

We now consider the term $$q \left( \frac{f_3^3}{f_1} \right) f_{21}.$$
Residues $a$ modulo $p^2$ that this series can reach must satisfy $$a \equiv \frac{63}{2}x^2 - \frac{21}{2}x + 3y^2 -2y + 1 \pmod{p^2}$$ \noindent for some $x, y \in \mathbb{Z}$.  Letting $b \equiv x-6^{-1} \pmod{p^2}$ and $c \equiv y - 3^{-1} \pmod{p^2}$, we find that $$a + 5 \cdot 24^{-1} \equiv \frac{63}{2} b^2 + 3c^2 \pmod{p^2}.$$

Thus if $a \equiv -5 \cdot 24^{-1} \pmod{p}$, we must have $$\frac{63}{2} b^2 + 3c^2 \equiv 0 \pmod{p}.$$
By the same logic, if $b, c \not\equiv 0 \pmod{p}$, then $-2 \cdot 21^{-1}$ must be a quadratic residue modulo $p^2$, and $-42$ must be as well.

If we are seeking residue classes avoided by \emph{both} terms of $\sum_{n=0}^\infty b_{21}(4n+1)q^n$, then we may choose to require that $-6$ be a quadratic nonresidue, and therefore, in order to make $-42$ a quadratic nonresidue, we further desire that 7 be a quadratic residue.  The latter occurs for primes congruent to 1, 3, 9, 19, 25, or 27 (mod 28).

A quick use of the Chinese Remainder Theorem now shows that both sets of conditions are satisfied for the primes modulo 168 listed in Theorem \ref{21in40}. Thus, these subprogressions of $\{b_{21}(4n)\}$ are identically zero modulo 2.

As for the proof of Theorem \ref{21in41}, we just note that, with its single term involved, completing the square yields that reachable residues $a$ modulo $p^2$ satisfy:
$$a + 11 \cdot 24^{-1} \equiv 3b^2 + \frac{9}{2}c^2 \pmod{p^2}.$$
The theorem now follows along the previous lines. \end{proof}

\section{Case $m=25$}\label{twentyfivesec}

Our proofs for $b_{25}$ are by dissection, with a similar flavor to those of Section \ref{twentyonesec}.

\begin{proof}[Proof of Theorem \ref{b25thm}]   We begin with the identity \cite{JKZ}, equation (4) (based on \cite{Hirsch}, equation (13), or \cite{BBG}, p. 301). Namely, $$f_1 f_5 \equiv {f_1}^6 + q {f_5}^6.$$

This gives us an even-odd dissection of $B_5$ as: $$\frac{f_5}{f_1} \equiv {f_1}^4 + q \frac{{f_5}^6}{{f_1}^2}.$$
Thus, \begin{equation}\label{b25dissect}B_{25} = \frac{f_{25}}{f_1} = \frac{f_{25}}{f_5} \frac{f_5}{f_1} $$$$\equiv {f_1}^4 {f_5}^4 + q^6 {f_5}^4 \frac{{f_{25}}^6}{{f_1}^2} + q \frac{{f_5}^{10}}{{f_1}^2} + q^5 {f_1}^4 \frac{{f_{25}}^6}{{f_5}^2} .\end{equation}

This easily yields that:
$$\sum_{n=0}^\infty b_{25} (2n) q^n \equiv {f_1}^2 {f_5}^2 + q^3 \frac{{f_5}^2 {f_{25}}^3}{f_1}$$
and
$$\sum_{n=0}^\infty b_{25} (2n+1) q^n \equiv \frac{{f_5}^{5}}{{f_1}} + q^2 {f_1}^2 \frac{{f_{25}}^3}{{f_5}}$$
$$\equiv {f_1}^4 {f_5}^4 + q \frac{{f_5}^{10}}{{f_1}^2} + q^2 {f_1}^2 {f_5}^4 {f_{25}}^2 + q^7 \frac{{f_1}^2 {f_{25}}^8}{{f_5}^2}.$$

Note that in going from the first to the second line, the first two terms are the dissection of $\frac{{f_5}^5}{f_1}$, while in the latter two terms we reserve factors $q^2 {f_1}^2 {f_{25}}^2$ and expand $\frac{f_{25}}{f_5}$ as the 5-magnification of $B_5$ just above.

Applying (\ref{b25dissect}) again to the factors $f_{25} / f_1$ in $\sum_{n=0}^\infty b_{25}(2n) q^n$, we can now extract the even and odd terms to obtain the following: 
\begin{align}\label{254}
\sum_{n=0}^\infty b_{25}(4n) q^n &\equiv f_1 f_5 + q^2 \frac{{f_5}^6 f_{25}}{f_1} + q^4 {f_1}^2 {f_{25}}^4; \\
\sum_{n=0}^\infty b_{25} (4n+2) q^n &\equiv q {f_1}^2 {f_5}^3 f_{25} + q^4 {f_5}^3 \frac{{f_{25}}^4}{f_1}.
\end{align}

Using equation (\ref{fivecore}), we further decompose $\sum_{n=0}^\infty b_{25}(4n) q^n$ as: 
$$\sum_{n=0}^\infty b_{25}(4n) q^n \equiv f_1 f_5 + q^4 {f_1}^2 {f_{25}}^4 $$$$+ q \left({f_5}^6 + q^5 {f_{25}}^6 \right) \left( \sum_{n=1}^\infty q^{n^2} + \sum_{n=1}^\infty q^{2n^2} + \sum_{n=1}^\infty q^{5n^2} + \sum_{n=1}^\infty q^{10n^2} \right).$$

Although there are ten terms to sum, each is congruent to a product of two quadratic series and so we have, by Lemma \ref{laclemma}, that all of them are lacunary modulo 2.  

We further establish the claimed zero progressions by analyzing their residues modulo $p^2$ for prime $p$.

For the term $f_1 f_5$ and any prime $p > 5$, we have that $$a \equiv \frac{1}{2}(3i^2-i) + \frac{5}{2}(3j^2 - j) \pmod{p^2},$$
which implies $$a + 4^{-1} \equiv \frac{3}{2} x^2 + \frac{15}{2} y^2 \pmod{p^2},$$
where $4^{-1}$ is taken modulo $p^2$.  By arguments similar to those of Section \ref{twentyonesec}, if $a \equiv -4^{-1} \pmod{p}$ and $x, y \not\equiv 0 \pmod{p}$, then $-5$ is a quadratic residue modulo $p$.  

If $(\frac{-5}{p}) = -1$ then $a \not\equiv kp-4^{-1} \pmod{p^2}$ unless $k \equiv 0 \pmod{p}$, and so this term avoids those progressions.  Therefore, we will find primes for which $(\frac{-5}{p}) = -1$.

For the term $$q^4 {f_1}^2 {f_{25}}^4,$$
similar calculations yield that if $a \equiv -4^{-1} \pmod{p}$ but not modulo $p^2$, then $-50$ is a quadratic residue modulo $p$, which is equivalent to $-2$ being a quadratic residue.  We take primes for which this and the previous condition are simultaneously false.

For the terms $$q {f_5}^6 \sum_{n=1}^\infty q^{An^2},$$
where $A \in \{1, 2, 5, 10\}$, we find that we desire that $-A/5$ not be a quadratic residue modulo $p$. For the terms
$$q^6 {f_{25}}^6 \sum_{n=1}^\infty q^{An^2},$$
we desire that $-A/25$ not be a quadratic residue.

In total, we want all of $-5$, $-2$, and $-1$ to be simultaneously nonresidues modulo $p$, in which case progressions $kp - 4^{-1}$ will be avoided modulo $p^2$ for all $k \not\equiv 0 \pmod{p}$.  

Note that this occurs for primes $p \equiv 31$ or $39 \pmod{40}$.  An example of this clause of the theorem is that $$b_{25}(4(31^2 n + 31k + 240)) \equiv 0  \, \quad \, \text{ for } \, 1 \leq k < 31.$$

For the $4n+2$ clause of the Theorem, we employ Ramanujan's identity (\cite{Ramanujan}, p. 212):
$$\sum_{n=0}^\infty p(5n+4) q^n = 5 \frac{{f_5}^5}{{f_1}^6}.$$
In particular, this tells us that
\begin{equation}\label{rara}
\frac{1}{f_1} \equiv q^4 \frac{{f_{25}}^5}{{f_5}^6} + \left( \dots \right),
\end{equation}
where the elided terms have powers that are not congruent to 4 (mod 5).

In $\sum_{n=0}^\infty b_{25}(4n+2) q^n$ we now repeatedly expand $f_1 f_5$ and its magnified forms, and replace the instance of $1/f_1$ with the right side of (\ref{rara}).  We obtain:
$$\sum_{n=0}^\infty b_{25}(4n+2) q^n \equiv q {f_1}^{12} {f_5}^6 + q^6 {f_1}^{12} {f_{25}}^6 + q^3 {f_5}^{18} + q^8 {f_5}^{12} {f_{25}}^6 $$$$+ q^4 ( {f_5}^{18} + q^5 {f_5}^{12} {f_{25}}^6 + q^{10} {f_5}^6 {f_{25}}^{12} + q^{15} {f_{25}}^{18} ) \cdot \left( q^4 \frac{{f_{25}}^6}{{f_5}^6} + \dots \right) .$$

We now wish to extract those terms that are 3 (mod 5). Note that
$${f_1}^{12} \equiv \sum_{n=0}^\infty q^{4 \binom{n+1}{2}}.$$ The  numbers $4\binom{n+1}{2}$ are  congruent to $2 \pmod{5}$ if and only if $n \equiv 2 \pmod{5}$, and therefore a little algebra gives us that $${f_1}^{12} \equiv q^{12} f_{25}^{12} + ( \dots ),$$ \noindent where the elided terms are not congruent to 2 (mod 5).

We thus obtain $$\sum_{n=0}^\infty b_{25} (20n+14) q^n \equiv q^2 {f_5}^{12} {f_1}^6 + q^3 {f_5}^{18} + {f_1}^{18}$$$$ + q {f_1}^{12} {f_5}^6 + q {f_1}^{12} {f_5}^6 + q^2 {f_1}^6 {f_5}^{12} + q^3 {f_5}^{18} + q^4 \frac{{f_5}^{24}}{{f_1}^6}.$$

Like terms cancel modulo 2, yielding
$$\sum_{n=0}^\infty b_{25} (20n+14) q^n \equiv {f_1}^{18} + q^4 \frac{{f_5}^{24}}{{f_1}^6}.$$
But this can be nonzero only for even $n$, giving us that
$$b_{25}(40n + 34) \equiv 0 .$$

For $b_{25}(4n+1)$, we start by extracting even terms from $b_{25}(2n+1)$, to get:
$$\sum_{n=0}^\infty b_{25}(4n+1) q^n \equiv {f_1}^2 {f_5}^2 + q f_1 {f_5}^2 f_{25}.$$
We now replace $f_1 {f_5}^2 f_{25}$ by $(f_1 f_5)(f_5 f_{25})$ and expand with the usual identity, canceling like terms modulo 2, to eventually obtain:
$$\sum_{n=0}^\infty b_{25}(4n+1) q^n \equiv {f_1}^{12} + q {f_1}^6 {f_5}^6 + q^6 {f_1}^6 {f_{25}}^6 + q^7 {f_5}^6 {f_{25}}^6.$$

Much of the argument is now analogous.  We may treat ${f_1}^{12}$ as ${f_1}^8 {f_1}^4$ to make the arguments consonant.  We find after analysis of all terms that we again desire $-1$, $-2$, and $-5$ to be quadratic residues modulo $p$ when $a \equiv -2^{-1} \pmod{p^2}$, and the clause follows.

Finally, the $8n+3$ clause of Theorem \ref{b25thm} is easier.  Extracting odd terms from $b_{25}(4n+3)$ yields:
$$\sum_{n=0}^\infty b_{25}(4n+3) q^n \equiv \frac{{f_5}^5}{f_1} + q^3 \frac{f_1 {f_{25}}^4}{f_5} $$$$\equiv {f_1}^4 {f_5}^4 + q \frac{{f_5}^{10}}{{f_1}^2} + q^3 \frac{{f_1}^6 {f_{25}}^4}{{f_5}^2} + q^4 {f_5}^4 {f_{25}}^4.$$

Extracting even terms now gives us:
$$\sum_{n=0}^\infty b_{25}(8n+3) q^n \equiv {f_1}^2 {f_5}^2 + q^2 {f_5}^2 {f_{25}}^2.$$  Finally, expanding each instance of $f_1 f_5$ and canceling like terms easily completes the proof.
\end{proof}

We note that, by an additional dissection within the progression $4n+1$, we may also establish the following theorem. While it only holds for progressions modulo $8(p^2 n + pk + \alpha) + 5$, it does hold for more primes.
\begin{theorem}\label{twentyfive8} If $p \equiv 11, 13, 17$, or $19 \pmod{20}$ is prime, then $$b_{25}(8(p^2n+pk - 3 \cdot 4^{-1}) + 5) \equiv 0 $$ for all $1 \leq k < p$, where $-3 \cdot 4^{-1}$ is taken modulo $p^2$.
\end{theorem}

Which of Theorem \ref{b25thm}, Clause (2) or Theorem \ref{twentyfive8} is stronger may be a matter of taste.

We also remark that, experimentally, similar statements seem to hold for progressions of the form $2(p^2 n + pk + \alpha) +1$.  Because these conjectures are often based on self-similarities, they should in principle be verifiable by the same modular forms machinery used elsewhere in this paper, although without further insight the computations might be even more extensive.  We conjecture the following:
\begin{conjecture} For a positive proportion of primes $p$, it holds that $$\sum_{n=0}^\infty b_{25} (2pn+\alpha) q^n \equiv q^{\beta} \sum_{n=0}^\infty b_{25}(2n+1) q^{pn},$$ \noindent for some $\alpha$ and $\beta$ depending on $p$.
\end{conjecture}

We note that from such congruences, and any of the base cases which we have shown to be identically zero, infinite families would follow of identically zero progressions with modulus $2p^{2k}$.  For instance, computational data suggests that:
$$\sum_{n=0}^\infty b_{25}(38n+37) q^n \equiv q^{18} \sum_{n=0}^\infty b_{25} (2n+1) q^{19n}.$$
This would establish the base case
$$b_{25}(722n+37) \equiv 0,$$
and from there the infinite family $$b_{25}(2 \cdot 19^{2k} n + 38 \cdot 19^{2k-2} -1) \equiv 0.$$

\section{Future Directions}

All proofs presented in the previous sections of this paper have established results for $m$-regular partitions, which are a particularly interesting class of eta-quotients.  However, it is possible for these results to have broader applicability if we can establish relations between the densities $\delta_m$ and those of other eta-quotients, or even define such densities as invariants of suitable infinite classes. 

As a sample result in this direction, here we provide a proof of Theorem \ref{b6thm} (see Section \ref{listing} for its statement).

\begin{proof}[Proof of Theorem \ref{b6thm}.] Applying identity (\ref{onethree}), we can quickly establish:
\begin{align*}\frac{f_6}{f_1} &\equiv \frac{{f_1}^8}{f_3} + q \frac{{f_3}^{11}}{{f_1}^4}; \\
\frac{{f_3}^{11}}{{f_1}^4} &\equiv {f_1}^5 {f_3}^8 + q \frac{{f_3}^{20}}{{f_1}^7}; \\
\frac{{f_3}^{20}}{{f_1}^7} &\equiv {f_1}^2 {f_3}^{17} + q \frac{{f_3}^{29}}{{f_1}^{10}};
\end{align*}

\noindent and in general, for $k \geq 3$, $$\frac{{f_3}^{9k+2}}{{f_1}^{3k+1}} \equiv \frac{{f_3}^{9k-1}}{{f_1}^{3k-8}}+ q \frac{{f_3}^{9k+11}}{{f_1}^{3k+4}}.$$

Since $$(9k-1)/3 \geq 3k-8$$ for all $k \geq 3$, and the hypothesis of Theorem \ref{cot} also applies for ${f_1}^8 / f_3$, ${f_1}^5 {f_3}^8$, and ${f_1}^2 {f_3}^{17}$, the first term on the right side of each equivalence has density zero. Thus, the left term has the same density as the second term on the right, assuming these densities exist.

It follows that if the density of any of the eta-quotients ${{f_3}^{9k+2}}/{{f_1}^{3k+1}}$ exists, then they all exist, and they are equal to the $k=0$ case, which is the density $\delta_6$ of $B_6$. (Note that we conjectured this latter to be $1/2$ in Conjecture \ref{noncong}.) \end{proof}

Among the most interesting notions that we think pertain to this topic is that of \emph{self-similarity}.  This behavior has been established for the partition function for some odd moduli.  For example, in \cite{FKO}, Folsom, Kent, and Ono proved:
\begin{theorem} Suppose that $5 \leq \ell \leq 31$ is prime and that $m \geq 1$. If $b$ is sufficiently large, then for every prime $c \geq 5$, there exists an integer $\lambda_\ell(m,c)$ such that, for all $n$ coprime to $c$: $$p \left((\ell^b n c^3 + 1)/24 \right) \equiv \lambda_\ell(m,c) p \left( (\ell^b n c + 1)/24 \right) \pmod{\ell^m}.$$
\end{theorem}

The authors of \cite{BLMT} proved analogous results for multipartitions, which are further cases of eta-quotients.  Their theorems again hold modulo powers of primes $\ell \geq 5$.  As for the parity of multipartitions, the closest results seem to come from \cite{JKZ} and \cite{JZ}, in which the present authors, along with Judge, proved congruences modulo 2 such as
$$q \sum_{n=0}^\infty p(19n+4) q^n \equiv \frac{1}{{f_1}^{19}} + \frac{1}{f_{19}},$$
and conjectured an infinite framework that contains identities of that form.  

As part of a program to understand the parity of the partition function and, more broadly, of eta-quotients, we suggest that it may be fruitful to develop general results along the following lines, in order to explain self-similarities of the types repeatedly encountered in this paper.  The following problems suggest themselves, forming a substantial line of research:
\begin{enumerate}
\item State conditions on primes $p$, integers $m$, shifts $k$, and pairs $(A,B)$ and $(C,D)$ with $A > C$, such that it generally holds that $$\sum_{n=0}^\infty b_m(An+B) q^n \equiv q^k \sum_{n=0}^\infty b_m(C n + D) q^{pn}.$$
\item As partial progress toward the above, prove Conjectures \ref{threeconj}, \ref{b9conj}, and \ref{b19conj}, and identify the primes involved in  Conjectures \ref{threeconj} and \ref{b19conj}.  These will suggest infinite families of self-similarities for $m \in \{3, 9, 19\}$. 
\item Find a less computationally intensive process to verify self-similarities such as those from Clause (2) individually (which might well lead to a fuller proof).
\item Prove Conjecture \ref{noncong}. The first clause, that there exists no $(A,B)$ for which $b_m(An+B) \equiv 0$ for the given values of $m$, seems more accessible.
\end{enumerate}

Of course, the latter clause of Conjecture \ref{noncong} on whether the given series have density $1/2$, and more generally, our main Conjecture \ref{mainconj}, are in our opinion of fundamental importance. In fact, solving them might very well require substantial advances of independent interest -- possibly, coming from the theory of modular forms -- in our basic understanding of eta-quotients.  

Finally, since it has never been proven that any eta-quotient has a nonzero odd density, we conclude by offering the suggestion that the next groundbreaking result in this line of investigation will be a proof that some eta-quotient does in fact have this property.

\section*{Acknowledgements} We thank Ken Ono for comments. Marie Jameson and Larry Rolen provided advice regarding a separate project, of which elements were useful background in this work.  Presentation comments by the referee are gratefully accepted.  The second author was partially supported by a Simons Foundation grant (\#630401).

\end{document}